%% file: arXiv.tex
\documentclass{article}
\usepackage{amsmath, amssymb, amsthm}
\usepackage{accents}
\usepackage{algorithm,algorithmic}
\usepackage{ascmac}
\usepackage{cases}
\usepackage{comment}
\usepackage{tikz-cd}
\usepackage[all]{xy}
\usepackage{multirow}
\usepackage{titlesec}
\usepackage{appendix}
\usepackage{hyperref}
\makeatletter
\def\sbar{\accentset{{\cc@style\underline{\mskip13mu}}}}
\def\mbar{\accentset{{\cc@style\underline{\mskip25mu}}}}
\makeatother

\theoremstyle{definition}
\newtheorem{Def}{Definition}[subsection]

\newtheorem{Lem}[Def]{Lemma}
\newtheorem{Prop}[Def]{Proposition}
\newtheorem{Thm}[Def]{Theorem}
\newtheorem{Cor}[Def]{Corollary}
\newtheorem{Exp}[Def]{Expectation}
\newtheorem{Rmk}[Def]{Remark}
\newtheorem{Alg}[Def]{Algorithm}
\newtheorem{Conj}{Conjecture}
\newtheorem{Theo}{Theorem}
\theoremstyle{plain}

\title{Computing superspecial hyperelliptic curves of genus 4 with automorphism group properly containing the Klein 4-group}
\author{Ryo Ohashi\thanks{Graduate School of Information Science and Technology, The University of Tokyo, E-mail: \texttt{ryo-ohashi@g.ecc.u-tokyo.ac.jp}} \and Momonari Kudo\thanks{Fukuoka Institute of Technology, E-mail: \texttt{m-kudo@fit.ac.jp}}}
\date{\today}

\begin{document}

\maketitle

\input{sec1}

\input{sec2}

\input{sec3}

\input{sec4}

\input{sec5}
\subsection*{Acknowledgments}
This research was conducted under a contract of “Research and development on new generation cryptography for secure wireless communication services" among “Research and Development for Expansion of Radio\,Wave Resources (JPJ000254)", which was supported by the Ministry of Internal Affairs and Communications, Japan.
This work was supported by JSPS Grant-in-Aid for Young Scientists 20K14301 and 23K12949.


\input{secA}

\end{document}

%% file: sec1.tex
\begin{abstract}
    In algebraic geometry, enumerating or finding superspecial curves in positive characteristic $p$ is important both in theory and in computation.
    In this paper, we propose feasible algorithms to enumerate or find superspecial hyperelliptic curves of genus $4$ with automorphism group properly containing the Klein $4$-group.
    Executing the algorithms on Magma, we succeeded in enumerating such superspecial curves for every $p$ with $19 \leq p < 500$, and in finding a single one for every $p$ with $19 \leq p < 7000$.
\end{abstract}

\section{Introduction}\label{sec:intro}
Throughout this paper, a curve means a non-singular projective variety of dimension one, and isomorphisms are ones over an algebraically closed field.
The complexity is measured by the number of (arithmetic) operations in a field, and the soft-O notation $\tilde{O}$ is used for ignoring logarithmic factors.
Let $p$ be a prime, and $k$ an algebraically closed field of characteristic $p$.
A curve $C$ is said to be {\it superspecial} if its Jacobian variety is isomorphic to a product of supersingular elliptic curves.
This is well-known to be equivalent to that the Cartier operator on the regular differential forms on $C$ acts as zero, where a matrix representing the operator with respect to a suitable basis is called a {\it Cartier-Manin matrix} of $C$.
For a given pair $(g,p)$, the number of isomorphism classes of superspecial genus-$g$ curves in characteristic $p$
is finite (possibly zero), and it is important to determine the (non-)existence of such a curve;
if exists, the next problem is to enumerate all isomorphism classes.
For $g \leq 3$, these problems were both solved in arbitrary $p$, but for $g\geq 4$ (even when $g=4$), they are still open in general $p$, see \cite{KudoRIMS} for a survey. 

The case on which we focus in this paper is $g=4$, and we here review known results in the case briefly; we refer to \cite{KudoRIMS} for details.
In this case, there does not exist any superspecial curve when $p \in \{2,3\}$ by Ekedahl's bound~\cite{Ekedahl}, while it follows from \cite[Theorem 3.1]{FGT} that such a curve uniquely (up to isomorphism) exists in $p=5$.
To investigate the case where $p \geq 7$, Kudo-Harashita 
developed {\it computational} approaches for enumeration and search of superspecial genus-4 curves, see \cite{KH17} and \cite{KH20} for algorithms in the non-hyperelliptic case, and \cite{KH18} and \cite{KH22} for ones in the hyperelliptic case.
The former (resp.\ latter) two algorithms can enumerate {\it all} superspecial non-hyperelliptic (resp.\ {hyperelliptic}) curves of genus $4$ for a given $p$, and its computational cost heavily depends on the Gr\"{o}bner basis computation.
Although the complexity of each algorithm is considered to be exponential with respect to $p$, their efficient implementations on Magma~\cite{Magma} established enumeration results for some small $p \leq 23$; 
a particular but the most surprising result is \cite[Theorem B]{KH17}, which says that there is no superspecial curve of genus $4$ in characteristic $p=7$.

On the other hand, some polynomial-time algorithms have been proposed in recent years, for restricted cases where curves have non-trivial automorphism groups: the paper
\cite{KHH} for the non-hyperelliptic case where ${\rm Aut} \supset \mathbf{V}_4$, and \cite{OKH} (resp.\ \cite{KNT}) for the hyperelliptic case where ${\rm Aut} \supset \mathbf{V}_4$ (resp.\ ${\rm Aut} \supset \mathbf{C}_6$), where $\mathbf{V}_4$ and $\mathbf{C}_n$ respectively denote the Klein $4$-group and the cyclic group of order $n$.
In particular,
Ohashi-Kudo-Harashita~\cite{OKH} proposed an algorithm to enumerate superspecial hyperelliptic curves $H$ of genus 4 such that ${\rm Aut}(H) \supset \mathbf{V}_4$, and its complexity is estimated as $O(p^4)$ operations in $\mathbb{F}_{p^4}$.
If we restrict ourselves to the case where $\mathrm{Aut}(H) \cong \mathbf{V}_4$, the algorithm is expected to terminate in $O(p^3)$.
 
In this paper, we explore the hyperelliptic case where $\mathrm{Aut}(H)$ {\it properly} contains $\mathbf{V}_4$, namely $\mathrm{Aut}(H) \supsetneq \mathbf{V}_4$.
In this case, it is known that $\mathrm{Aut}(H)$ is isomorphic to $\mathbf{D}_4$, $\mathbf{D}_8$, $\mathbf{D}_{10}$, $\mathbf{C}_{16} \rtimes \mathbf{C}_2$, or $\mathbf{C}_5 \rtimes \mathbf{D}_4$, where $\mathbf{D}_n$ denotes the dihedral group of order $2n$.
Our contributions are summarized as follows:
\begin{itemize}
    \item In Section \ref{subsec:D4}, we propose an algorithm (Algorithm \ref{alg:D4} below) to enumerate superspecial hyperelliptic curves $H$ of genus $4$ with $\mathrm{Aut}(H) \supset \mathbf{D}_4$ in characteristic $p$, as a simpler version of Ohashi-Kudo-Harashita's algorithm~\cite{OKH}.
    Thanks to our criterion (Lemma \ref{lem:D4} below) to detect $H$'s satisfying $\mathrm{Aut}(H) \supset \mathbf{D}_4$ together with Proposition \ref{lem:SelfCol}, the complexity of the proposed algorithm is estimated as $O(p^3)$ operations in $\mathbb{F}_{p^4}$, which is more efficient than directly applying Ohashi-Kudo-Harashita's one.
    \item In Sections \ref{subsec:D8} and \ref{subsec:D10}, we also propose efficient algorithms specific to the case where $\mathrm{Aut}(H) \supset \mathbf{D}_8$ or $\mathrm{Aut}(H) \supset \mathbf{D}_{10}$, by using Cartier-Manin matrices together with Gauss' hypergeometric series.
    For each case, we explicitly determine the Cartier-Manin matrix of $H$, and prove that counting all superspecial $H$'s is just done by computing the gcd of only one pair of univariate polynomials.
    With this fact together with our isomorphy criteria, the complexity of enumerating superspecial curves by this algorithm is $\tilde{O}(p^2)$ operations in $\mathbb{F}_{p^4}$, and it is expected to be $\tilde{O}(p)$ in practice.
    \item Implementing and executing the proposed algorithms on Magma, we obtain the following computational results:
\end{itemize}

\begin{Theo}\label{ThmB}
    For each type of automorphism group properly containing $\mathbf{V}_4$, the number of superspecial hyperelliptic curves of genus $4$ in characteristic $p$ with $19 \leq p < 500$ are summarized in Table \ref{table:new} below. 
    Moreover, such a superspecial curve exists for every prime $p$ with $19 \leq p < 7000$.
\end{Theo}

As in \cite{OKH}, the upper bounds on $p$ in Theorem \ref{ThmB} can be also updated easily, see Section \ref{sec:comp} for details.
From our computational results, we conjecture (cf.\ Conjecture \ref{conj} below) that there exists a superspecial hyperelliptic curve $H$ of genus $4$ with $\mathrm{Aut}(H) \supset \mathbf{D}_4$ in arbitrary characteristic $p$ with $p \geq 19$.

%% file: sec2.tex
\section{Preliminaries}
In this section, let $k$ be an algebraically closed field of characteristic $p \geq 3$.

\subsection{Cartier-Manin matrix of a certain hyperelliptic curve}
Let $g \geq 2$ be an integer. Any hyperelliptic curve $H$ over $k$ of genus $g$ is isomorphic to
\[
    H: y^2 = f(x), \quad \deg{f(x)} = 2g+1\ {\rm or}\ 2g+2
\]
where $f(x) \in k[x]$ is a square-free polynomial. Then, the {\it Cartier-Manin matrix\,} of $H$ is defined to be the $(g \times g)$-matrix
\[
    M = \begin{pmatrix}
        \gamma_{p-1} & \gamma_{p-2} & \cdots & \gamma_{p-g}\\
        \gamma_{2p-1} & \gamma_{2p-2} & \cdots & \gamma_{2p-g}\\
        \vdots & \vdots & \ddots & \vdots\\
        \gamma_{gp-1} & \gamma_{gp-2} & \cdots & \gamma_{gp-g}
    \end{pmatrix},
\]
where we write
\[
    f(x)^{(p-1)/2} = \sum_{i=0}^n \gamma_ix^i, \quad\ n := \frac{p-1}{2} \cdot\, \deg{f(x)}
\]
with $\gamma_i \in k$. The Cartier-Manin matrix $M$ enables us to determine whether $H$ is superspecial or not by the following result:
\begin{Thm}[{\cite[Theorem 4.1]{Nygaard}}]\label{thm:Nygaard}
With notation as above, $H$ is superspecial if and only if $M$ is zero.
\end{Thm}

In the following, we consider the hyperelliptic curves
\[
    C: y^2 = (x^r-1)(x^r-\lambda) =: f(x), \quad D: y^2 = x(x^r-1)(x^r-\lambda) =: g(x),
\]
where $\lambda \in k$. We denote respectively by $\alpha_d$ and $\beta_d$ the $x^d$-coefficients of $f(x)^e$ and $g(x)^e$ with $e := (p-1)/2$.
We also use the following notation:
\begin{Def}
A {\it truncation of Gauss' hypergeometric series} is defined as
\[
    G^{(d)}(a,b,c\,;z) := \sum_{n=0}^d \frac{(a\,;n)(b\,;n)}{(c\,;n)(1\,;n)}z^n,
\]
where $(x\,;n) := x(x+1) \cdots (x+n-1)$.
\end{Def}

We can show that each entry of the Cartier-Manin matrices of $C$ and $D$ is written as a truncation of Gauss' hypergeometric series:
\begin{Prop}\label{alpha}
For all integers $d$ with $0 \leq d \leq r(p-1)$, we obtain the following statements:
\begin{enumerate}
\item If $d \not\equiv 0 \pmod{r}$, then we have $\alpha_d = \beta_{d+e} = 0$.
\item If $d = kr$ with $0 \leq k \leq e$, then we have
\[
    \alpha_d = \beta_{d+e} \equiv \lambda^{e-k} \cdot (-1)^k\binom{e}{k}G^{(k)}(1/2,-k,1/2-k\,;\lambda) \pmod{p}.
\]
\item If $d = kr$ with $e \leq k \leq p-1$, then we have
\[
    \alpha_d = \beta_{d+e} \equiv (-1)^k\binom{e}{p-1-k}G^{(k)}(1/2,1+k,3/2+k\,;\lambda) \pmod{p}.
\]
\end{enumerate}
\end{Prop}
\begin{proof}
Letting $\gamma_k$ be the $x^k$-coefficient of $\{(x-1)(x-\lambda)\}^{(p-1)/2}$, we have that $\alpha_d = \beta_{d+e}$ is equal to $\gamma_k$.
The assertions follow from Proposition \ref{gamma} below.
\end{proof}
\begin{Cor}\label{divisible}
For all integers $k$ with $0 \leq k \leq e$, we have
\[
    \alpha_{(e-k)r} \equiv \alpha_{(e+k)r}\lambda^k\ \,{\rm and}\ \,\beta_{(e-k)r+e} \equiv \beta_{(e+k)r+e}\lambda^k.
\]
\end{Cor}
\begin{proof}
It follows from Proposition \ref{alpha} and Corollary \ref{div}.
\end{proof}

\subsection{Hyperelliptic curves of genus 4}

We recall a classification by Kudo-Nakagawa-Takagi~\cite{KNTpre} (a preliminary version of \cite{KNT}) of hyperelliptic curves of genus $4$ by means of their automorphism groups.
For a hyperelliptic curve $H$ over $k$, we denote by $\iota_H$ its hyperelliptic involution.
The quotient group $\overline{\rm Aut}(H):=\mathrm{Aut}(H) / \langle \iota_H \rangle$ is called the {\it reduced automorphism group} of $H$.
For $n \geq 2$, we also denote respectively by $\mathbf{C}_n$, $\mathbf{D}_n$, $\mathbf{A}_n$, $\mathbf{V}_4$, and $\mathbf{Q}_8$ the cyclic group of order $n$, the dihedral group of order $2n$, the alternating group of order $n!/2$, the Klein $4$-group $\mathbf{C}_2 \times \mathbf{C}_2$, and the quaternion group.

\begin{Thm}[{\cite[Theorem C]{KNTpre}}]\label{thm:app}
Assume that $p \geq 7$.
The reduced automorphism group $\overline{\mathrm{Aut}}(H)$ of a hyperelliptic curve $H$ of genus $4$ over an algebraically closed field $k$ of characteristic $p$ is isomorphic to either of the $10$ finite groups listed in Table \ref{table:aut}.
In each type of $\overline{\mathrm{Aut}}(H)$, the finite group isomorphic to $\mathrm{Aut}(H)$ is provided in the third column of Table \ref{table:aut}, and the hyperelliptic curve $H$ is isomorphic to $y^2 = f(x)$ given in the fourth column of the table.
\end{Thm}

\renewcommand{\arraystretch}{1.4}
\begin{table}[H]
\centering{
\caption{Classification of hyperelliptic curves $H$ of genus $4$ over an algebraically closed field $k$ of characteristic $p \geq 7$ according to their reduced and full automorphism groups, with their defining equations $y^2 = f(x)$, where $A,B,C,D \in k$.}
\label{table:aut}
\vspace{5pt}
\scalebox{0.85}{
\begin{tabular}{c||c|c|l} \hline
Type & $\overline{\mathrm{Aut}}(H)$ & ${\mathrm{Aut}}(H)$ & Defining equation $y^2 = f(x)$ for $H$ (in normal form)  \\ \hline 
{\bf 1} & $\{ 0 \}$ & $\mathbf{C}_2$ & $y^2 = (\mbox{square-free polynomial in $x$ of degree 9 or 10})$ \\ \hline
{\bf 2-1} & $\mathbf{C}_2$ &  $\mathbf{V}_4$ & $y^2 = x^{10} + A x^8 + B x^6 + C x^4 + D x^2 + 1$  \\ \hline
{\bf 2-2} & $\mathbf{C}_2$ & $\mathbf{C}_4$ & $y^2 = x^{9} + A x^7 + B x^5 + C x^3 + x$ \\ \hline
{\bf 3} & $\mathbf{C}_3$ & $\mathbf{C}_6$ & $y^2 = x^{10} + A x^7 + B x^4 + x$  \\ 
\multirow{2}{*}{\bf 4-1} & \multirow{2}{*}{$\mathbf{V}_4$} & \multirow{2}{*}{$\mathbf{D}_4$} & $y^2 = x^{10} + A x^8 + B x^6 + B x^4 + A x^2 + 1$, or   \\ 
 &  &  & $y^2 = x^{9} + A x^7 + B x^5 + A x^3 + x$  \\ \hline
{\bf 4-2} & $\mathbf{V}_4$ & $\mathbf{Q}_8$ & $y^2 = x (x^4-1)(x^4+A x^2+1)$ \\ 
{\bf 5} & $\mathbf{D}_4$ & $\mathbf{D}_8$ & $y^2 = x^{9} + A x^5 + x$ \\ 
{\bf 6} & $\mathbf{D}_5$ & $\mathbf{D}_{10}$ & $y^2 = x^{10} + A x^5 + 1$ \\ \hline
{\bf 7} & $\mathbf{A}_4$ & $\mathrm{SL}_2(\mathbb{F}_3)$ & $y^2 = x (x^4-1)(x^4 + 2 \sqrt{-3} x^2 + 1)$  \\ 
{\bf 8} & $\mathbf{D}_8$ & $\mathbf{C}_{16} \rtimes \mathbf{C}_2$ & $y^2 = x^{9} + x$\\ 
{\bf 9} & $\mathbf{C}_9$ & $\mathbf{C}_{18}$ & $y^2 = x^{10} + x$  \\
{\bf 10} & $\mathbf{D}_{10}$ & $\mathbf{C}_5 \rtimes \mathbf{D}_{4}$ & $y^2 = x^{10}+1$ \\ \hline
\end{tabular}
}
}
\end{table}
\renewcommand{\arraystretch}{1}

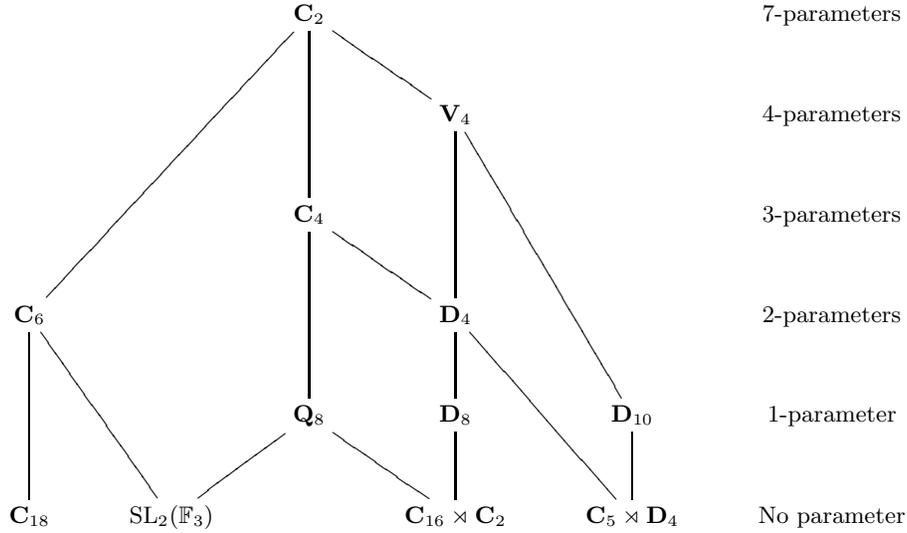
\begin{figure}[H]
\label{fig:aut}
 \centering
\caption{Containment relationships among the families of genus-$4$ hyperelliptic curves $H$ with given automorphism groups ${\rm Aut}(H)$, and the number of parameters representing each family.}
\renewcommand{\arraystretch}{0.8}
\centering{{\small 
    \[
        \xymatrix{
        & &  \mathbf{C}_2  \ar@{-}[rd] \ar@{-}[dd] \ar@{-}[llddd]  & & & \mbox{7-parameters}\\
        & & & \mathbf{V}_4 \ar@{-}[dd]  \ar@{-}[rddd]&   & \mbox{4-parameters}\\
        & & {\mathbf{C}_4} \ar@{-}[dd] \ar@{-}[rd]   & & & \mbox{3-parameters}\\
        \mathbf{C}_6 \ar@{-}[dd] \ar@{-}[rdd]  & &  \ar@{-}[d] & {\mathbf{D}_4} \ar@{-}[rdd] \ar@{-}[d] &  & \mbox{2-parameters}\\
        & & {\mathbf{Q}_8} \ar@{-}[ld]\ar@{-}[rd]  &  {\mathbf{D}_8} \ar@{-}[d]    & {\mathbf{D}_{10}} \ar@{-}[d] & \mbox{1-parameter}\\
        \mathbf{C}_{18} & {\mathrm{SL}_2(\mathbb{F}_3)} & & \mathbf{C}_{16} \rtimes \mathbf{C}_2  & \mathbf{C}_5 \rtimes \mathbf{D}_4 & \mbox{No parameter}}
    \] 
}}
\end{figure}

Note that the superspeciality of the curves of Types {\bf 8}, {\bf 9}, and {\bf 10} can be described as follows:
\begin{Cor}\label{spssp}
Assume that $p \geq 7$. Then, we have the following statements:
\begin{enumerate}
\item The genus-4 hyperelliptic curve $H$ whose automorphism group ${\rm Aut}(H)$ is isomorphic to $\mathbf{C}_{16} \rtimes \mathbf{C}_2$ is superspecial if and only if $p \equiv -1,9 \pmod{16}$.
\item The genus-4 hyperelliptic curve $H$ whose automorphism group ${\rm Aut}(H)$ is isomorphic to $\mathbf{C}_{18}$ is superspecial if and only if $p \equiv 8 \pmod{9}$.
\item The genus-4 hyperelliptic curve $H$ whose automorphism group ${\rm Aut}(H)$ is isomorphic to $\mathbf{C}_5 \rtimes \mathbf{D}_{4}$ is superspecial if and only if $p \equiv 9 \pmod{10}$.
\end{enumerate}
\end{Cor}
\begin{proof}
The curves in the cases (1) and (2) are isomorphic to
\[
    y^2 = x^9 - x \quad {\rm and } \quad y^2 = x^9 - 1
\]
respectively, and hence the assertions for these cases are shown by \cite[Theorem 2]{Robert}.
Moreover, the curve in the case (3) is isomorphic to $H: y^2 = x^{10} - 1$, which has two involutions
\[
    \sigma: (x,y) \mapsto (-x,y), \quad {\rm and} \quad \tau: (x,y) \mapsto (-x,-y).
\]
Then, the quotient curves $C_1 := H/\langle\sigma\rangle$ and $C_2 := H/\langle\tau\rangle$ are given by
\[
    C_1: y_1^2 = x^5 - 1 \quad {\rm and} \quad C_2: y_2^2 = x(x^5-1)
\]
respectively, and these curves are both superspecial if and only if $p \equiv 9 \pmod{10}$ by \cite{IKO}. Since the Jacobian of $H$ is $(2,2,2,2)$-isogenous to $J(C_1) \times J(C_2)$, we have that $H$ is superspecial if and only if $p \equiv 9 \pmod{10}$.
\end{proof}


\subsection{Review on results by Ohashi-Kudo-Harashita}\label{subsec:V4}

As noted in Section \ref{sec:intro}, the computational enumeration of superspecial curves is an important problem, and Ohashi-Kudo-Harashita~\cite{OKH} proposed an algorithm specific to the case of genus-$4$ hyperelliptic curves $H$ with $\mathrm{Aut}(H) \supset \mathbf{V}_4$, i.e., types {\bf 2-1}, {\bf 4-1}, {\bf 5}, {\bf 6}, {\bf 8}, and {\bf 10} in Table \ref{table:aut}.
Their idea is to reduce the enumeration of superspecial $H$'s into that of special pairs of superspecial genus-$2$ curves obtained as the quotients of $H$ by non-hyperelliptic involutions.

In the following, we briefly review the idea and their enumeration algorithm.
The target curves are of the form $y^2 = f(x^2)$ for some square-free polynomial $f(x) \in k[x]$ of degree $5$, as in type {\bf 2-1} of Table \ref{table:aut}.
Each $H$ of the target curves has non-hyperelliptic involutions $(x,y) \mapsto (-x,y)$ and $(x,y) \mapsto (-x,-y)$, whence the quotients of $H$ by them are given by $y_1^2 = f(x)$ and $y_2^2 = x f(x)$ respectively.
This implies that any genus-$4$ hyperelliptic curve with automorphism group containing $\mathbf{V}_4$ is realized as the normalization of the fiber product over $\mathbb{P}^1$ of two genus-$2$ curves sharing exactly $5$ ramification points, and it is straightforward that the converse also holds.
Moreover, transforming $2$ among ramification points of $H$ into $\{\pm 1\}$, we may assume that $H$ is of the following normal form
\begin{equation}\label{eq:hyp}
	H: y^2 =f(x^2)= (x^2-1)(x^2-c_2)(x^2-c_3)(x^2-c_4)(x^2-c_5)
\end{equation}
with $f(x)=(x-1)(x-c_2)(x-c_3)(x-c_4)(x-c_5)$ for some $c_i \in k \!\smallsetminus\! \{ 0,1\}$.
Moreover, we can transform the two quotient curves into the following Rosenhain forms:
\begin{align}
     C_1 &: y_1^2 = f_1(x) = x(x-1)(x-\lambda_1)(x-\lambda_2)(x-\lambda_3), \label{eq:C1}\\
     C_2 &: y_2^2 = f_2(x) = x(x-1)(x-\lambda_1)(x-\lambda_2)(x-\lambda_3'),\label{eq:C2}
\end{align}
via $x \mapsto \frac{(x-c_2)(c_3-1)}{(x-1)(c_3-c_2)}$,
where $\lambda_1$, $\lambda_2$, $\lambda_3$, and $\lambda_3'$ are mutually distinct elements in $k \!\smallsetminus\! \{ 0, 1 \}$ 
given by 
\begin{equation}\label{eq:lambdai}
\begin{split}
    \lambda_1 \!=\! \frac{(c_4-c_2)(c_3-1)}{(c_4-1)(c_3-c_2)}, \ \lambda_2\!=\! \frac{(c_5-c_2)(c_3-1)}{(c_5-1)(c_3-c_2)},\
    \lambda_3 \!=\! \frac{c_2(c_3-1)}{c_3-c_2}, \ \lambda_3'\!=\! \frac{c_3-1}{c_3-c_2}.
\end{split}
\end{equation}
Namely, $1$, $c_2$, $c_3$, $c_4$, $c_5$, $0$, and $\infty$ are mapped into $\infty$, $0$, $1$, $\lambda_1$, $\lambda_2$, $\lambda_3$, and $\lambda_3'$ respectively.
Note that, given $\lambda_1$, $\lambda_2$, $\lambda_3$, and $\lambda_3'$, we can reconstruct $c_2$, $c_3$, $c_4$, and $c_5$ through the map $x \mapsto \frac{x-\lambda_3'}{x-\lambda_3}$, explicitly
\begin{equation}\label{eq:ci}
    	c_2 = \frac{\lambda_3'}{\lambda_3}, \quad c_3 = \frac{1-\lambda_3'}{1-\lambda_3}, \quad c_4 = \frac{\lambda_1-\lambda_3'}{\lambda_1-\lambda_3}, \quad c_5 = \frac{\lambda_2-\lambda_3'}{\lambda_2-\lambda_3}.
\end{equation}
Note that exchanging $\lambda_1$ and $\lambda_2$ produces the same equation \eqref{eq:hyp}.
Exchanging $\lambda_3$ and $\lambda_3'$, we have $y^2 = (x^2-1)(x^2-c_2^{-1})(x^2-c_3^{-1})(x^2-c_4^{-1})(x^2-c_5^{-1})$, which is isomorphic to \eqref{eq:hyp} via $(x,y) \mapsto (1/x,\sqrt{-c_2c_3c_4c_5} y/x^5)$.
{Namely, $H$ is specified by a pair $(\{ \lambda_1,\lambda_2\},\{\lambda_3,\lambda_3'\})$ of the sets $\{ \lambda_1,\lambda_2\}$ and $\{ \lambda_3,\lambda_3'\}$.}

Recall from \cite[Lemma 2]{OKH} that $H$ is superspecial if and only if $C_1$ and $C_2$ are both superspecial.
Based on the above discussion, Ohashi-Kudo-Harashita~\cite{OKH} presented the following algorithm to enumerate superspecial $H$'s:

\begin{Alg}[{\cite[Algorithm 1]{OKH}}]\label{alg:V4}
    ~
\begin{description}
\item[{\it Input:}] A rational prime $p\geq 7$.
\item[{\it Output:}] A list of representatives for the isomorphism classes of superspecial genus-$4$ hyperelliptic curves $H$ over $\overline{\mathbb{F}_p}$ with $\mathrm{Aut}(H) \supset \mathbf{V}_4$.
\end{description}
\begin{description}
\item[Step 1.]
Enumerate all isomorphism classes of superspecial curves of genus $2$ in characteristic $p$, and let $\mathsf{SSp}_2(p)$ be a list of the enumerated superspecial genus-$2$ curves.

\item[Step 2.] For each superspecial genus-$2$ curve $C$ in $\mathsf{SSp}_2(p)$, compute all $C_{\lambda,\mu,\nu}$ in Rosenhain form with $C_{\lambda,\mu,\nu} \cong C$, by translating $3$ among $6$ ramified points of $C \to {\mathbb P}^1$ to $\{0,1,\infty\}$.
Sorting the list of all genus-$2$ curves in Rosenhain form generated as above, detect all pairs $(C_{\lambda,\mu,\nu},C_{\lambda',\mu',\nu'})$ with $\# (\{ \lambda,\mu,\nu \} \cap \{ \lambda',\mu',\nu'\} ) = 2$.
Let $\mathsf{HowePairsList}$ be a list of the computed pairs.

\item[Step 3.] Classify isomorphism classes of $H$'s as in \eqref{eq:hyp} corresponding via \eqref{eq:ci} to the pairs $(C_1,C_2)$ in $\mathsf{HowePairsList}$.
Finally, output a list of the computed isomorphism classes.
\end{description}
\end{Alg}

{
Ohashi-Kudo-Harashita estimated in \cite{OKH} that the complexity of Algorithm \ref{alg:V4} is ${O}(p^4)$ operations in $\mathbb{F}_{p^4}$.
More precisely, the complexities of Steps 1, 2, and 3 are estimated as $O(p^3)$, $O(p^3)$, and $O(p^4)$ respectively.
We can estimate in fact that the complexity of Step 3 is $O(p^4)$ operations in $\mathbb{F}_{p^2}$ as follows:
Ohashi's recent result~\cite[Main Theorem A]{Ohashi} implies that $\lambda, \mu, \nu \in \mathbb{F}_{p^2}$ for any superspecial $C_{\lambda,\mu,\nu}$ computed in Step 2. 
In Step 3, we need to test whether two hyperelliptic curves of genus $4$ are isomorphic or not in many times, and it would be expected from \cite[Remark 4]{OKH} that the number of required isomorphy tests is $O(p^4)$.
In general, testing isomorphy of two hyperelliptic curves of genus $4$ is done by computing a Gr\"{o}bner basis (with respect to a fast monomial ordering) for a polynomial system with $5$ variables and maximal total degree $10$, and its complexity is $O(1)$ operations in any common base field of the two curves, see \cite[Section 2.1]{OKH} for details.
Therefore, the complexity of Step 3 is $O(p^4)$ operations in $\mathbb{F}_{p^2}$.
If we restrict ourselves to the case where $\mathrm{Aut}(H) \cong \mathbf{V}_4$ in Step 3, the complexity of Step 3 is $O(p^3)$ operations in $\mathbb{F}_{p^2}$, so that the total complexity is $O(p^3)$ operations in $\mathbb{F}_{p^2}$.}

Ohashi-Kudo-Harashita implemented Algorithm \ref{alg:V4} on Magma, and succeeded in enumerating superspecial $H$'s with $\mathrm{Aut}(H)\cong \mathbf{V}_4$ for $p$ up to $200$.

%% file: sec3.tex
\section{Main results}

In this section, we first construct a feasible algorithm for enumerating superspecial genus-$4$ hyperelliptic curves $H$ with $\mathrm{Aut}(H) \supset \mathbf{D}_4$, as a simpler version of Algorithm \ref{alg:V4}.
We also present more efficient algorithms specific to the cases where $\mathrm{Aut}(H) \supset \mathbf{D}_8$ and $\mathrm{Aut}(H) \supset \mathbf{D}_{10}$, by using Cartier-Manin matrices together with the theory of Gauss' hypergeometric series.

\subsection{The case where ${\rm Aut} \cong {\mathbf D}_4$}\label{subsec:D4}

Let $H$ be a hyperelliptic curve of genus $4$ over $k$ with $\mathrm{Aut}(H) \supset \mathbf{D}_4$.
A naive way to enumerate superspecial $H$'s is to proceed with the following:
Apply Steps 1 and 2 of Algorithm \ref{alg:V4}, detect $H$ with $\mathrm{Aut}(H) \supset \mathbf{D}_4$ from the list $\mathsf{HowePairsList}$ by computing the automorphism group of each curve, and then classify the isomorphism classes of detected ones as in the original Step 3 in $O(p^4)$ operations in $\mathbb{F}_{p^4}$.
However, in fact, we can simplify Step 2 (e.g., we need not compute any automorphism group), and moreover the complexity of the isomorphism classification can be estimated as $O(p^2)$.
A key lemma for simplifying Step 2 is the following:

\begin{Lem}\label{lem:D4}
With notation as above, $H$ is birational to $C_1 \times_{\mathbb{P}^1} C_2$ for some genus-$2$ curves $C_1$ and $C_2$ over $k$ with $C_1 \cong C_2$ that share exactly $5$ ramification points.
In fact, we can take equations for $C_1$ and $C_2$ to be ones in Rosenhain form $C_{\lambda_1,\lambda_2,\lambda_3}$ and $C_{\lambda_1,\lambda_2,\lambda_3'}$ with $\lambda_3 \neq \lambda_3'$ satisfying $c_2 c_3 = c_4 c_5 = 1$, where $c_2$, $c_3$, $c_4$, and $c_5$ are given in \eqref{eq:ci}.

Conversely, if $C_{\lambda_1,\lambda_2,\lambda_3}$ and $C_{\lambda_1,\lambda_2,\lambda_3'}$ with $\lambda_3 \neq \lambda_3'$ satisfy $c_2 c_3 = c_4 c_5 = 1$,
then they are isomorphic to each other, and $ C_{\lambda_1,\lambda_2,\lambda_3}\times_{\mathbb{P}^1} C_{\lambda_1,\lambda_2,\lambda_3'}$ is birational to a hyperelliptic curve of genus $4$ whose automorphism group contains $\mathbf{D}_4$.
\end{Lem}

\begin{proof}
    Recall from Theorem \ref{thm:app} that $H$ is isomorphic to
    \begin{align}
    y^2 &= x^{10} + A x^8 + B x^6 + B x^4 + A x^2 + 1,\ \mbox{or} \label{eq:D4type1}\\
    y^2 &= x^{9} + A x^7 + B x^5 + A x^3 + x \label{eq:D4type2}
    \end{align}
    for some $A, B \in k$.
    By Lemma \ref{lem:type1type2} below, it suffices to consider the first equation \eqref{eq:D4type1} only, which is transformed into
    \begin{align*}
        y^2 &= x^{10} - A x^8 + B x^6 - B x^4 + A x^2 -1\\
        &= (x^2-1)(x^8 + (1-A)x^6 + (-A + B + 1) x^4 + (1-A)x^2 + 1)
    \end{align*}
    by $(x,y) \to (\sqrt{-1}x,\sqrt{-1}y/x^5)$.
    Putting $1-A = a+b$ and $-A+B-1 = a b$, we can write the right hand side as
    \[
    (x^2-1)(x^4 + a x^2 + 1)(x^4 + b x^2 + 1) = (x^2-1)(x^2 - c_2)(x^2- c_3)(x^2- c_4)(x^2-c_5)
    \]
    for some $c_2, c_3, c_4, c_5 \in k \!\smallsetminus\! \{ 0, 1 \}$ with $c_2c_3 = c_4 c_5 = 1$.
    As in Subsection \ref{subsec:V4}, $H$ is birational to the fiber product $C_1 \times_{\mathbb{P}^1} C_2$ of the two genus-$2$ curves $C_1: y_1^2 = f(x) = (x-1)(x-c_2)(x-c_3)(x-c_4)(x-c_5)$ and $C_2 : y_2^2 = x f(x)$,
     which share $5$ ramification points precisely.
     By $c_2c_3 = c_4 c_5 = 1$, these $C_1$ and $C_2$ are isomorphic to each other via the map $(x,y) \mapsto (1/x,\sqrt{-c_2c_3c_4c_5}y/x^3)$.
     Moreover, they are isomorphic to \eqref{eq:C1} and \eqref{eq:C2} respectively, via \eqref{eq:lambdai} and \eqref{eq:ci}.
     We have proved the first assertion, and it is straightforward that the second assertion (the converse of the first assertion) holds.
\end{proof}

\begin{Lem}\label{lem:type1type2}
    Any genus-$4$ hyperelliptic curve given by \eqref{eq:D4type2} is isomorphic to one given by \eqref{eq:D4type1}.
\end{Lem}
\begin{proof}
    As in \eqref{eq:D4type1}, the right hand side of \eqref{eq:D4type2} is factored into
     \[
    x(x^4 - a x^2 + 1)(x^4 - b x^2 + 1) = x(x^2 - c^2)(x^2- c^{-2})(x^2-d^2)(x^2-d^{-2})
    \]
    for some $a,b \in k \!\smallsetminus\! \{ \pm 2 \}$ with $a \neq b$ and some $c,d \in k \!\smallsetminus\! \{ 0,\pm 1\}$ with $c \neq d$.
    Since the M\"{o}bius transformation $\phi : x \to \frac{1+x}{1-x}$ maps $\{ \infty, 0, \pm c, \pm c^{-1}, \pm d, \pm d^{-1} \}$ to $\{ \pm 1, \pm \xi, \pm \xi^{-1}, \pm \eta, \pm \eta^{-1} \}$ with $\xi := \phi (c) $ and $\eta := \phi(d)$, the genus-$4$ hyperelliptic curve given by \eqref{eq:D4type2} is isomorphic to
    \[
    y^2 = (x^2-1)(x^2 - \xi^2)(x^2 - \xi^{-2}) (x^2 - \eta^2)(x^2 - \eta^{-2}).
    \]
    Expanding the right hand side, we obtain an equation of the form \eqref{eq:D4type1}, as required.
\end{proof}

From Lemma \ref{lem:D4}, we not necessarily consider all genus-$2$ curves in Rosenhain form constructed in Step 2 of Algorithm \ref{alg:V4}:
It suffices to consider only Rosenhain forms $C_{\lambda_1,\lambda_2,\lambda_3}$ constructed from the {\it same} $C_{\lambda,\mu,\nu}$.
In addition, we may discard any pair of $C_{\lambda_1,\lambda_2,\lambda_3}$ and $C_{\lambda_1,\lambda_2,\lambda_3'}$ with $\lambda_3 \neq \lambda_3'$ if $(c_2,c_3,c_4,c_5)$ given in \eqref{eq:ci} does not satisfy $c_2 c_3 = c_4 c_5 = 1$.
This also means that we can detect pairs $(C_1,C_2)$ such that associated $H$ satisfies $\mathrm{Aut}(H) \supset \mathbf{D}_4$ {\it without} computing $\mathrm{Aut}(H)$.
Here, we write down an algorithm to enumerate superspecial $H$'s with $\mathrm{Aut}(H) \supset \mathbf{D}_4$, as a simpler version of Algorithm \ref{alg:V4}.

\begin{Alg}\label{alg:D4}
    ~
\begin{description}
\item[{\it Input:}] A rational prime $p\geq 7$.
\item[{\it Output:}] A list of representatives for the isomorphism classes of superspecial genus-$4$ hyperelliptic curves $H$ over $\overline{\mathbb{F}_p}$ with $\mathrm{Aut}(H) \supset \mathbf{D}_4$.
\end{description}
\begin{description}
\item[Step 1.] Proceed with the same as in Step 1 of Algorithm \ref{alg:V4}, but
we give a representative of each isomorphism class in Rosenhain form.
Let $\mathsf{SSp}_2(p)$ be a list of enumerated superspecial genus-$2$ curves in Rosenhain form.

\item[Step 2.] 
For each superspecial genus-$2$ curve $C_{\lambda,\mu,\nu}$, generate all $C_{\lambda',\mu',\nu'}$ with $\lambda',\mu',\nu' \in k \smallsetminus \{0,1\}$ such that $C_{\lambda,\mu,\nu} \cong C_{\lambda',\mu',\nu'}$, by translating $3$ among the $6$ ramified points of $C_{\lambda,\mu,\nu}\to {\mathbb P}^1$ to $\{0,1,\infty\}$.
Among $C_{\lambda',\mu',\nu'}$'s produced as above, detect ones satisfying $\# (\{ \lambda,\mu,\nu \} \cap \{ \lambda',\mu',\nu'\} ) = 2$.
For each detected one, when we write $C_{\lambda,\mu,\nu}$ and $C_{\lambda',\mu',\nu'}$ as $C_{\lambda_1,\lambda_2,\lambda_3}$ and $C_{\lambda_1,\lambda_2,\lambda_3'}$ respectively, compute an equation \eqref{eq:hyp} defining a hyperelliptic curve $H$ birational to $C_{\lambda_1,\lambda_2,\lambda_3} \times_{\mathbb{P}^1} C_{\lambda_1,\lambda_2,\lambda_3'}$ by \eqref{eq:ci}, and store it if $c_2 c_3 = c_4 c_5 = 1$.
Let $\mathsf{ListSSpHypCurvesD4}$ be a list of $H$'s computed as above.

\item[Step 3.] Classify isomorphism classes of $H$'s in $\mathsf{ListSSpHypCurvesD4}$.
Finally, output a list of the computed isomorphism classes.
\end{description}
\end{Alg}

\begin{Rmk}
    {In our implementation, we realize Step 1 by generating the superspecial Richelot isogeny graph, which can be computed with arithmetic of theta functions~\cite{DMPR}.
    Since this is not the main topic of this paper, we omit to write its concrete procedures, and some details will be explained in \cite{OOKYN}.}

    We also remark that our implementation conducts each isomorphism test in Step 3 with Magma's built-in function \texttt{IsIsomorphicHyperellipticCurves}, which was implemented by Lercier-Sijsling-Ritzenthaler~\cite{LSR21}.
\end{Rmk}






The correctness of Algorithm \ref{alg:D4} follows from Lemma \ref{lem:D4}.
The complexity of Step 1 is the same as that of Algorithm \ref{alg:V4}, i.e., $O(p^3)$ operations in $\mathbb{F}_{p^4}$.
Note that, for each $C_{\lambda,\mu,\nu}$ computed in Step 1, we have $\lambda, \mu,\nu \in \mathbb{F}_{p^2}\!\smallsetminus\!\{ 0,1\}$ by \cite[Main Theorem A]{Ohashi}.
Therefore, the complexity of Step 2 is $O(p^3)$ operations in $\mathbb{F}_{p^2}$.
To estimate the complexity of Step 3, we use the following lemma:

\begin{Prop}\label{lem:SelfCol}
    In Step 2 of Algorithm \ref{alg:D4}, if $\lambda_1$, $\lambda_2$, and $\lambda_3 $ are chosen uniformly at random from $\mathbb{F}_{p^2} \!\smallsetminus\! \{ 0, 1 \}$, then the probability that $C_{\lambda_1,\lambda_2,\lambda_3} \cong C_{\lambda_1,\lambda_2,\lambda_3'}$ for some $\lambda_3' \in \mathbb{F}_{p^2}$ with $\lambda_3 \neq \lambda_3'$ is $O(1/p^2)$.
\end{Prop}

\begin{proof}
As in \cite[Section 2.1]{Ohashi}, we have that $C_{\lambda_1,\lambda_2,\lambda_3} \cong C_{\lambda,\mu,\nu}$ if and only if
\[
    \lambda = \frac{(a_4-a_1)(a_2-a_3)}{(a_4-a_3)(a_2-a_1)},\ \mu = \frac{(a_5-a_1)(a_2-a_3)}{(a_5-a_3)(a_2-a_1)},\ \nu = \frac{(a_6-a_1)(a_2-a_3)}{(a_6-a_3)(a_2-a_1)},
\]
where $\{a_1,\ldots,a_6\} = \{0,1,\infty,\lambda_1,\lambda_2,\lambda_3\}$.
Here, if $a_i = \infty$, we mean that all the factors $(a_i-a_j)$ or $(a_j-a_i)$ are excluded from each of $\lambda$, $\mu$, and $\nu$.
Then, we see that
\[
    \lambda = \lambda_1, \quad \mu = \frac{\lambda_1}{\lambda_2}, \quad \nu = \frac{\lambda_1}{\lambda_3}
\]
in the case where $(a_1,\ldots,a_6) = (\infty,\lambda_1,0,1,\lambda_2,\lambda_3)$. 
Let us compute the probability that
\[
    \#(\{\lambda_2,\lambda_3\} \cap \{\mu,\nu\}) = 1.
\]
Then, we have that $\mu = \lambda_2$ or $\nu = \lambda_3$ (indeed, if $\mu = \lambda_3$, then we also obtain $\nu = \lambda_2$, which contradicts the above condition). Therefore, we need only consider the case that $\lambda_1 = \lambda_2^2$ or $\lambda_1 = \lambda_3^2$, and we see that this probability is $O(1/p^2)$ since all $\lambda_1,\lambda_2,\lambda_3$ belong to $\mathbb{F}_{p^2}\!\smallsetminus\!\{0,1\}$.


There exist two other transformations
\begin{align*}
    \lambda &= \lambda_1, \quad \mu = \frac{\lambda_2-\lambda_1}{\lambda_2-1}, \quad \nu = \frac{\lambda_3-\lambda_1}{\lambda_3-1},\\
    \lambda &= \lambda_1, \quad \mu = \lambda_1 \cdot \frac{\lambda_2-1}{\lambda_2-\lambda_1}, \quad \nu = \lambda_1 \cdot \frac{\lambda_3-1}{\lambda_3-\lambda_1}
\end{align*}
satisfying $\lambda = \lambda_1$ except for the trivial transformation. In these cases, the probability that $\#(\{\lambda_2,\lambda_3\} \cap \{\mu,\nu\}) = 1$ is also $O(1/p^2)$ as similar discussions.
\end{proof}
\begin{Rmk}
For a fixed $3$-tuple $(\lambda_1,\lambda_2,\lambda_3)$, the number of $\lambda'_3$ such that $C_{\lambda_1,\lambda_2,\lambda_3} \cong C_{\lambda_1,\lambda_2,\lambda_3'}$ is $O(1)$ since the number of transformations is $O(1)$.
\end{Rmk}

Since the number of superspecial genus-$2$ curves is $O(p^3)$ by \cite[Theorem 3.3]{IKO}, we obtain the following expectation from Proposition \ref{lem:SelfCol}:

\begin{Exp}\label{exp:D4}
The number of pairs $(C_{\lambda_1,\lambda_2,\lambda_3}, C_{\lambda_1,\lambda_2,\lambda_3'})$ of superspecial $C_{\lambda_1,\lambda_2,\lambda_3}$ and $C_{\lambda_1,\lambda_2,\lambda_3'}$ with $\lambda_3 \neq \lambda_3'$ is $O(p)$.
\end{Exp}

{
Assuming Expectation \ref{exp:D4}, we can estimate the complexity of Step 3 in Algorithm \ref{alg:D4} is $O(p^2)$ operations in $\mathbb{F}_{p^2}$ similarly to the case of Algorithm \ref{alg:V4}, so that the total complexity becomes $O(p^3)$ operations in $\mathbb{F}_{p^4}$.}

\subsection{The case where ${\rm Aut} \cong {\mathbf D}_8$}\label{subsec:D8}

In this subsection, we shall present an algorithm specific to the case where $\mathrm{Aut}(H) \supset \mathbf{D}_8$.
For this, we start with proving the following lemma:

\begin{Lem}\label{lem:D8form}
Any genus-4 hyperelliptic curve $H$ such that ${\rm Aut}(H) \supset \mathbf{D}_8$ can be written as
\[
    H_\lambda: y^2 = x(x^4-1)(x^4-\lambda)
\]
for some $\lambda \neq 0,1$.
\end{Lem}
\begin{proof}
Recall from Table \ref{table:aut} that $H$ is written as
\[
    H: y^2 = x^9 + Ax^5 + x = x(x^4-\alpha^4)(x^4-\alpha^{-4}),
\]
where $\alpha^4 + \alpha^{-4} = -A$ with $\alpha^4 \neq 0,\pm 1$. Setting $\lambda := \alpha^8$, we obtain an isomorphism
\[
    H \rightarrow H_\lambda\ ;\,(x,y) \mapsto (\alpha x,\alpha^{9/2}y),
\]
where $\alpha^8 \neq 0,1$ implies $\lambda \neq 0,1$, as desired.
\end{proof}

\begin{Rmk}\label{Alambda}
    {In the proof of Lemma \ref{lem:D8form}, the genus-$4$ hyperelliptic curve $ H : y^2 = x^9 + A x^5 + x$ is isomorphic to $H_{\lambda}$ for some $\lambda \in k \!\smallsetminus\! \{ 0,1 \}$ and vice versa, via the relation $\lambda + \frac{1}{\lambda} = A^2 -2$.}
\end{Rmk}

We shall use the above $H_{\lambda}$ for our purpose that we enumerate or find superspecial curves. Next, we investigate the Cartier-Manin matrix of $H_\lambda$. Write
\[
    \bigl\{\hspace{0.2mm}x(x^4-1)(x^4-\lambda)\bigr\}^e = \sum_{i=e}^{9e} \beta_ix^i, \quad e := \frac{p-1}{2},
\]
where $\beta_i \in k$, and then we obtain that $\beta_i = 0$ unless $i-e$ is multiple of 4.
By a straightforward computation, we can determine possibly non-zero entries in the Cartier-Manin matrix $M$ of $H_{\lambda}$, dividing the case into the following four cases, by the value of $p \bmod 8$:

\begin{itemize}
\item If $p \equiv 1 \pmod{8}$, then $e \equiv 0 \pmod{4}$, and we have
\[
    M = \begin{pmatrix}
        \beta_{p-1} & 0 & 0 & 0\\
        0 & \beta_{2p-2} & 0 & 0\\
        0 & 0 & \beta_{3p-3} & 0\\
        0 & 0 & 0 & \beta_{4p-4}
    \end{pmatrix}.
\]
\item If $p \equiv 3 \pmod{8}$, then $e \equiv 1 \pmod{4}$, and we have
\[
    M = \begin{pmatrix}
        0 & \beta_{p-2} & 0 & 0\\
        \beta_{2p-1} & 0 & 0 & 0\\
        0 & 0 & 0 & \beta_{3p-4}\\
        0 & 0 & \beta_{4p-3} & 0
    \end{pmatrix}.
\]
\item If $p \equiv 5 \pmod{8}$, then $e \equiv 2 \pmod{4}$, and we have
\[
    M = \begin{pmatrix}
        0 & 0 & \beta_{p-3} & 0\\
        0 & 0 & 0 & \beta_{2p-4}\\
        \beta_{3p-1} & 0 & 0 & 0\\
        0 & \beta_{4p-2} & 0 & 0
    \end{pmatrix}.
\]
\item If $p \equiv 7 \pmod{8}$, then $e \equiv 3 \pmod{4}$, and we have
\[
    M = \begin{pmatrix}
        0 & 0 & 0 & \beta_{p-4}\\
        0 & 0 & \beta_{2p-3} & 0\\
        0 & \beta_{3p-2} & 0 & 0\\
        \beta_{4p-1} & 0 & 0 & 0
    \end{pmatrix}.
\]
\end{itemize}

Here, recall from Theorem \ref{thm:Nygaard} that $H$ is superspecial if and only if $M = 0$, which can be determined by computing the gcd of four possible non-zero entries (as univariate polynomials in $\lambda$) given as above. In fact, it suffices to consider only two among the four entries in each case:

\begin{Prop}\label{D8ssp}
With notation as above, we have the following statements:
\begin{itemize}
\item If $p \equiv 1 \pmod{8}$, then $H_\lambda$ is superspecial if and only if $\beta_{3p-3} = \beta_{4p-4} = 0$.
\item If $p \equiv 3 \pmod{8}$, then $H_\lambda$ is superspecial if and only if $\beta_{3p-4} = \beta_{4p-3} = 0$.
\item If $p \equiv 5 \pmod{8}$, then $H_\lambda$ is superspecial if and only if $\beta_{3p-1} = \beta_{4p-2} = 0$.
\item If $p \equiv 7 \pmod{8}$, then $H_\lambda$ is superspecial if and only if $\beta_{3p-2} = \beta_{4p-1} = 0$.
\end{itemize}
\end{Prop}
\begin{proof}
It follows from Corollary \ref{divisible} that
\[
    \beta_{5e-4i} \equiv \beta_{5e+4i}\lambda^i \hspace{-0.3mm}\pmod{p}
\]
for all $i \in \{0,\ldots,e\}$. In the case where $p \equiv 1 \pmod{8}$, we see that
\[
    \beta_{p-1} \equiv \beta_{4p-4}\lambda^{3e/4}\ \ {\rm and}\ \ \beta_{2p-2} \equiv \beta_{3p-3}\lambda^{e/4}.
\]
By the condition that $\lambda \neq 0$, we have
\[
    \beta_{p-1} = \beta_{2p-2} = \beta_{3p-3} = \beta_{4p-4} = 0 \ \Longleftrightarrow\,\beta_{3p-3} = \beta_{4p-4} = 0.
\]
The other three cases can be shown in a way similar to this.
\end{proof}

Putting $F := {\rm gcd}(\beta_{3p-i},\beta_{4p-j})$ for
\begin{equation}\label{D8ij}
    (i,j) := \left\{
    \begin{array}{l}
        (3,4) \quad {\rm if}\ p \equiv 1 \hspace{-2.5mm}\pmod{8},\\[0.5mm]
        (4,3) \quad {\rm if}\ p \equiv 3 \hspace{-2.5mm}\pmod{8},\\[0.5mm]
        (1,2) \quad {\rm if}\ p \equiv 5 \hspace{-2.5mm}\pmod{8},\\[0.5mm]
        (2,1) \quad {\rm if}\ p \equiv 7 \hspace{-2.5mm}\pmod{8},
    \end{array}
    \right.
\end{equation}
we have that $H_\lambda$ is superspecial if and only if $F(\lambda) = 0$.

Here, we have shown that a list of all superspecial $H_{\lambda}$ can be generated by computing the gcd $F(\lambda)$ and by seeking its roots, all of which are contained in $\mathbb{F}_{p^4}\hspace{-0.5mm}$ (see Proposition \ref{D8belong} below).
Note that the number of listed superspecial curves $H_{\lambda}$ is $O(p)$ since $\deg{F} = O(p)$.
In fact, we can expect from our computational results that this degree (and thus the number of listed curves) is bounded by a constant not depending on $p$, see Table \ref{table:new} below.

Once the list is obtained, we execute the isomorphism classification to count the number of isomorphism classes of listed superspecial curves.
As we described in Subsection \ref{subsec:V4}, testing isomorphy of two hyperelliptic curves of genus $4$ is done in a constant number of operations (with respect to $p$) in their common base field.
Hence, the complexity of the isomorphism classification is $O(p^2)$ operations in $\mathbb{F}_{p^4}$ naively.
To make this more efficient, we prove the following criteria (Proposition \ref{D8iso} below), and then obtain a formula (Theorem \ref{D8thm} below) on the number of isomorphism classes:

\begin{Prop}\label{D8iso}
For $\lambda_1,\lambda_2 \in k \!\smallsetminus\! \{ 0,1 \}$, the curve $H_{\lambda_1}\hspace{-0.3mm}$ is isomorphic to the curve $H_{\lambda_2}\hspace{-0.3mm}$ if and only if $\lambda_1 = \lambda_2$ of $\lambda_1 = 1/\lambda_2$.
\end{Prop}
\begin{proof}
Write the two hyperelliptic curves $H_{\lambda_1}$ and $H_{\lambda_2}$ as follows:
\begin{align*}
    H_{\lambda_1} &: y^2 = x(x^4-\alpha^4)(x^4-\alpha^{-4}), \quad {\rm with}\ \alpha^8 = \lambda_1,\\
    H_{\lambda_2} &: y^2 = x(x^4-\beta^4)(x^4-\beta^{-4}), \quad \hspace{0.1mm}{\rm with}\ \beta^8 = \lambda_2.
\end{align*}
Then, one can show that $H_{\lambda_1} \cong H_{\lambda_2}$ if and only if
\[
    \lambda_1 + \frac{1}{\lambda_1} = \alpha^8 + \frac{1}{\alpha^8} = \beta^8 + \frac{1}{\beta^8} = \lambda_2 + \frac{1}{\lambda_2}
\]
by a way similar to the proof of \cite[Theorem 2.1]{Ishii}.
It is straightforward that this is equivalent to $\lambda_1 = \lambda_2$ or $\lambda_1 = 1/\lambda_2$, as desired.
\end{proof}

Recall from Theorem \ref{thm:app} (especially Figure \ref{fig:aut}) that $\mathrm{Aut}(H_{\lambda})$ is isomorphic to $\mathbf{D}_{8}$ or $\mathbf{C}_{16} \rtimes \mathbf{C}_2$; the latter case is easily detected as follows:

\begin{Cor}\label{D8special}
The following statements are true:
\begin{enumerate}
    \item If $\lambda = -1$, then we have ${\rm Aut}(H_\lambda) \cong \mathbf{C}_{16} \rtimes \mathbf{C}_2$.\vspace{-0.5mm}
    \item Otherwise, we have ${\rm Aut}(H_\lambda) \cong \mathbf{D}_{8}$.
\end{enumerate}
\end{Cor}
\begin{proof}
The assertion (1) follows from Remark \ref{Alambda} and the case {\bf 8} of Table \ref{table:aut}.
For any $\lambda \neq -1$, it follows from Proposition \ref{D8iso} that $H_{\lambda}$ is not isomorphic to $H_{-1}$, as desired.
\end{proof}

Thanks to the following theorem, we can compute the number of isomorphism classes of superspecial hyperelliptic genus-4 curves such that ${\rm Aut} \hspace{-0.1mm}\cong {\mathbf D}_8$. 
\begin{Thm}\label{D8thm}
With notation as above, the degree of $F$ is odd if and only if $p \equiv -1,9 \pmod{16}$, and the number of superspecial hyperelliptic curves $H$ of genus 4 such that ${\rm Aut}(H) \cong {\mathbf D}_8$ is given as\vspace{-1mm}
\[
    \left\{
    \begin{array}{l}
        \frac{1}{2}(\deg{F}-1) \quad {\rm if}\ p \equiv -1,9 \pmod{16},\\[1mm]
        \frac{1}{2}\deg{F} \hspace{4.7mm}\qquad {\rm otherwise.}
    \end{array}
    \right.
\]
\end{Thm}
\begin{proof}
{We first claim that $F$ is separable as a univariate polynomial in $\lambda$.}
{Indeed, each entry $\beta_k$ in the Cartier-Manin matrix for $H_{\lambda}$ is written as a Gauss' hypergeometric series by Corollary \ref{alpha}.
Then, one can show that all $\beta_i$'s (as univariate polynomials in $\lambda$) are separable and have no root $\lambda=0$ nor $\lambda =1$, by a way similar to the proof of \cite{Igusa}, \cite[$\S 1.4$]{IKO}, or \cite[Lemma 3.1.3]{C6}.}
{Therefore, $F$ is also separable.}

{Next, we prove the assertions.
For any $\lambda_0 \in k \!\smallsetminus \!\{ 0,1 \}$, it follows from Proposition \ref{D8iso} that $F(\lambda_0) = 0$ is equivalent to $F(\frac{1}{\lambda_0}) = 0$.
Here, $\lambda_0 = \frac{1}{\lambda_0}$ implies $\lambda_0 = \pm 1$ and vice versa, and $F(-1) = 0$ holds if and only if $p \equiv -1,9 \pmod{16}$ by Corollary \ref{spssp} (1) and Corollary \ref{D8special}.
From these together with the separability of $F$, we can write
\[
    F = 
    \left\{
    \begin{array}{ll}
        c(\lambda+1)^i \prod_{\lambda_0} (\lambda-\lambda_0)(\lambda-\frac{1}{\lambda_0}) & \quad {\rm if}\ p \equiv -1,9 \pmod{16},\\[0.5mm]
        c \prod_{\lambda_0} (\lambda-\lambda_0)(\lambda-\frac{1}{\lambda_0}) & \quad \mbox{otherwise}
    \end{array}
    \right.
\]
for a constant $c \in k \smallsetminus \{ 0 \}$ and some mutually different $\lambda_0$'s with $\lambda_0 \in k \smallsetminus \{ 0,\pm 1 \}$, which implies the assertion on the degree of $F$.}
{With the above factorization of $F$, we obtain the second 
 assertion from Proposition \ref{D8iso} and Corollary \ref{D8special}.}
\end{proof}

Moreover, we obtain the following proposition, which is used for determining the complexity of computing the list of superspecial curves $H_{\lambda}$ with exact $\lambda$:

\begin{Prop}\label{D8belong}
All roots of the polynomial  $F$ belong to $\mathbb{F}_{p^4}$.
\end{Prop}
\begin{proof}
It is known that any superspecial curve is defined over $\mathbb{F}_{p^2}$.
Hence, if $H_\lambda$ is superspecial, then there exists a curve $H$ defined over $\mathbb{F}_{p^2}$ such that $H_\lambda \cong H$.
Let ${H_\lambda}^{\!(\sigma)}\hspace{-0.2mm}$ be the fiber product $H_\lambda \otimes_{k,\sigma} k$, where $\sigma: k \rightarrow k$ is the $p^2$-power map, so that we have ${H_\lambda}^{\!(\sigma)} \cong H$. Since $H_{\sigma(\lambda)} \cong {H_\lambda}^{\!(\sigma)}$, we obtain that ${H_\lambda} \cong H_{\sigma(\lambda)}$. It follows from Proposition \ref{D8iso} that
\[
    \lambda +  \frac{1}{\lambda} = \sigma(\lambda) +  \frac{1}{\sigma(\lambda)} = \sigma\biggl(\lambda + \frac{1}{\lambda}\biggr),
\]
and hence $\lambda + 1/\lambda$ belongs to $\mathbb{F}_{p^2}$.
{This implies that $\lambda$ and $1/\lambda$ are the roots of a quadratic polynomial over $\mathbb{F}_{p^2}$, so that $\lambda \in \mathbb{F}_{p^4}$, as desired.}
\end{proof}

Putting the results obtained in this subsection together,
we obtain the following algorithm:
\begin{Alg}\label{alg:D8}
    ~
\begin{description}
\item[{\it Input:}] A rational prime $p\geq 7$.
\item[{\it Output:}] A list $\mathcal{L}$ of representatives for the isomorphism classes of superspecial genus-$4$ hyperelliptic curves $H$ over $\overline{\mathbb{F}_p}$ with $\mathrm{Aut}(H) \cong \mathbf{D}_8$.
\end{description}
\begin{description}
\item[Step 1.] Compute $\beta_{3p-i}$ and $\beta_{4p-j}$ by using Proposition \ref{alpha} {(details will be explained in the proof of Theorem \ref{D8comp} below)}, where $(i,j)$ denotes a pair of integers defined in (\ref{D8ij}).
\item[Step 2.] {Compute the polynomial $F := {\rm gcd}(\beta_{3p-i},\beta_{4p-j})$ and its roots $\lambda \in \mathbb{F}_{p^4}$.}
\item[Step 3.] For each $\lambda \neq -1$ obtained in Step 2, if not $H_{1/\lambda}$ already belongs to $\mathcal{L}$, then set $\mathcal{L} \leftarrow \mathcal{L} \hspace{0.5mm}\cup \{H_\lambda\}$. Finally, output a list $\mathcal{L}$.
\end{description}
\end{Alg}

We finally determine the complexity of Algorithm \ref{alg:D8} in the following:

\begin{Thm}\label{D8comp}
We can compute the number of superspecial genus-$4$ hyperelliptic curves $H$ with $\mathrm{Aut}(H) \cong \mathbf{D}_8$ in $\tilde{O}(p)$ operations in $\mathbb{F}_{p}$.
{
Moreover, the complexity of listing up all of their isomorphism classes is bounded by $\tilde{O}(p+d^2)$ operations in $\mathbb{F}_{p^4}$ with $d := \deg{F}=O(p)$, and it becomes $\tilde{O}(p)$ operations in $\mathbb{F}_{p^4}$ assuming that $\deg{F}$ is bounded by a constant not depending on $p$.}
\end{Thm}

\begin{proof}
By definition of $\beta_i$, we see that the degrees of $\beta_{3p-i}$ and $\beta_{4p-j}$ are both $O(p)$.
These are computed in $\tilde{O}(p)$ operations in $\mathbb{F}_p$.
Indeed, we can write
\begin{align*}
    \beta_{3p-i} &= (-1)^{k_1}\binom{e}{p-1-k_1}G^{(k_1)}(1/2,1+k_1,3/2+k_1\,;\lambda), \ \ k_1 = \frac{3p-i-e}{4},\\
    \beta_{4p-j} &= (-1)^{k_2}\binom{e}{p-1-k_2}G^{(k_2)}(1/2,1+k_2,3/2+k_2\,;\lambda), \ \ k_2 = \frac{4p-j-e}{4}
\end{align*}
by Proposition \ref{alpha}.
Here, Gauss' hypergeometric series $G^{(d)}(a,b,c\,;\lambda)$ can be computed as follows:
We set $g_0 := 1$ and {compute}
\[
    g_n :=\frac{(a+n-1)(b+n-1)}{(c+n-1)n} \cdot g_{n-1}
\]
{for each $n$ from $1$ to $d$}.
Then, the sum $g_0 + g_1\lambda + \cdots + g_d\lambda^d$ is equal to $G^{(d)}(a,b,c\,;\lambda)$.
Hence, the computation of Step 1 can be done in {${O}(p)$} operations in $\mathbb{F}_p$.

In Step 2, the polynomial $F$ can be also computed in $\tilde{O}(p)$ operations in $\mathbb{F}_p$ with the fast gcd algorithm, {since the degrees of $\beta_{3p-i}$ and $\beta_{4p-j}$ as univariate polynomials in $\lambda$ are both $O(p)$.
Hence,} the first assertion is shown by Theorem \ref{D8thm}.
If one computes a list of representatives of all isomorphism classes, it is required to compute all the roots in $\mathbb{F}_{p^4}\hspace{-0.3mm}$ of $F$.
By the separability of the polynomial $F$, we can compute the roots only by equal-degree factorization, e.g., the Cantor-Zassenhaus algorithm, whose complexity is ${O}(d^2\log{d}\log{p^4})$ (and thus $\tilde{O}(d^2)$) operations in $\mathbb{F}_{p^4}$ in this case.

It follows from Proposition \ref{D8iso} that computations in Step 3 can be done in $O(d \log{d})$ operations in $\mathbb{F}_{p^4}$.
Considering the complexity of each step together, we obtain the second assertion. {When we suppose $d=O(1)$, it is clear that the complexity of Step 2 (resp.\ Step 3) is $\tilde{O}(p)$ (resp.\ $O(1)$) operations in $\mathbb{F}_{p^4}$, which deduces the last assertion.}
\end{proof}

\subsection{The case where ${\rm Aut} \cong {\mathbf D}_{10}$}\label{subsec:D10}
Similarly to the previous subsection, we shall present an algorithm specific to the case where $\mathrm{Aut}(H) \supset \mathbf{D}_{10}$.
We first prove the following lemma:

\begin{Lem}\label{lem:D10form}
Any genus-4 hyperelliptic curve $H$ such that ${\rm Aut}(H) \supset \mathbf{D}_{10}$ can be written as
\[
    {H'}_{\!\!\lambda}: y^2 = (x^5-1)(x^5-\lambda)
\]
for some $\lambda \neq 0,1$.
\end{Lem}
\begin{proof}
From Table \ref{table:aut}, we can write $H$ as
\[
    H: y^2 = x^{10} + Ax^5 + 1 = (x^5-\alpha^5)(x^5-\alpha^{-5})
\]
where $\alpha^5 + \alpha^{-5} = -A$ with $\alpha^5 \neq 0,\pm 1$.
Setting $\lambda := \alpha^{10}$, there exists an isomorphism
\[
    H \rightarrow {H'}_{\!\!\lambda}\ ;\,(x,y) \mapsto (\alpha x,\alpha^5y),
\]
where $\lambda \neq 0, 1$ by $\lambda = \alpha^{10} \neq 0,1$, as desired.
\end{proof}
\begin{Rmk}\label{Alambda10}
    {In the proof of Lemma \ref{lem:D10form}, the genus-$4$ hyperelliptic curve $ H : y^2 = x^{10} + A x^5 + 1$ is isomorphic to $H_{\lambda}$ for some $\lambda \in k \!\smallsetminus\! \{ 0,1 \}$ and vice versa, via the relation $\lambda + \frac{1}{\lambda} = A^2 -2$.}
\end{Rmk}

As for the Cartier-Manin matrix $M$ of ${H'}_{\!\!\lambda}$, one has
\[
    \bigl\{\hspace{0.2mm}(x^5-1)(x^5-\lambda)\bigr\}^e = \sum_{i=0}^{10e} \alpha_ix^i,
\]
where $e := (p-1)/2$.
We see that $\alpha_i = 0$ unless $i$ is multiple of 5.
A straightforward computation determines possibly non-zero entries in $M$, dividing the case into the following four cases, by the value of $p \bmod 10$:
\begin{itemize}
\item If $p \equiv 1 \pmod{10}$, then $e \equiv 0 \pmod{5}$, and we have
\[
    M = \begin{pmatrix}
        \alpha_{p-1} & 0 & 0 & 0\\
        0 & \alpha_{2p-2} & 0 & 0\\
        0 & 0 & \alpha_{3p-3} & 0\\
        0 & 0 & 0 & \alpha_{4p-4}
    \end{pmatrix}.
\]
\item If $p \equiv 3 \pmod{10}$, then $e \equiv 1 \pmod{5}$, and we have
\[
    M = \begin{pmatrix}
        0 & 0 & \alpha_{p-3} & 0\\
        \alpha_{2p-1} & 0 & 0 & 0\\
        0 & 0 & 0 & \alpha_{3p-4}\\
        0 & \alpha_{4p-2} & 0 & 0
    \end{pmatrix}.
\]
\item If $p \equiv 7 \pmod{10}$, then $e \equiv 3 \pmod{5}$, and we have
\[
    M = \begin{pmatrix}
        0 & \alpha_{p-2} & 0 & 0\\
        0 & 0 & 0 & \alpha_{2p-4}\\
        \alpha_{3p-1} & 0 & 0 & 0\\
        0 & 0 & \alpha_{4p-3} & 0
    \end{pmatrix}.
\]
\item If $p \equiv 9 \pmod{10}$, then $e \equiv 4 \pmod{5}$, and we have
\[
    M = \begin{pmatrix}
        0 & 0 & 0 & \alpha_{p-4}\\
        0 & 0 & \alpha_{2p-3} & 0\\
        0 & \alpha_{3p-2} & 0 & 0\\
        \alpha_{4p-1} & 0 & 0 & 0
    \end{pmatrix}.
\]
\end{itemize}

As in Proposition \ref{D8ssp}, we obtain the following proposition:

\begin{Prop}
With notation as above, we have the following:
\begin{itemize}
\item If $p \equiv 1 \pmod{10}$, then ${H'}_{\!\!\lambda}$ is superspecial if and only if $\alpha_{3p-3} \!=\! \alpha_{4p-4} \!=\! 0$.\vspace{-0.5mm}
\item If $p \equiv 3 \pmod{10}$, then ${H'}_{\!\!\lambda}$ is superspecial if and only if $\alpha_{3p-4} \!=\! \alpha_{4p-2} \!=\! 0$.\vspace{-0.5mm}
\item If $p \equiv 7 \pmod{10}$, then ${H'}_{\!\!\lambda}$ is superspecial if and only if $\alpha_{3p-1} \!=\! \alpha_{4p-3} \!=\! 0$.\vspace{-0.5mm} 
\item If $p \equiv 9 \pmod{10}$, then ${H'}_{\!\!\lambda}$ is superspecial if and only if $\alpha_{3p-2} \!=\! \alpha_{4p-1} \!=\! 0$.
\end{itemize}
\end{Prop}

\begin{proof}
This can be shown similarly to the proof of Proposition \ref{D8ssp}, and so let us omit to write a proof here.
\end{proof}

We define $G := {\rm gcd}(\alpha_{3p-i},\alpha_{4p-j})$, where a pair $(i,j)$ is given as
\begin{equation}\label{D10ij}
    (i,j) := \left\{
    \begin{array}{l}
        (3,4) \quad {\rm if}\ p \equiv 1 \hspace{-2.5mm}\pmod{10},\\
        (4,2) \quad {\rm if}\ p \equiv 3 \hspace{-2.5mm}\pmod{10},\\
        (1,3) \quad {\rm if}\ p \equiv 7 \hspace{-2.5mm}\pmod{10},\\
        (2,1) \quad {\rm if}\ p \equiv 9 \hspace{-2.5mm}\pmod{10}.
    \end{array}
    \right.
\end{equation}
Hence, the curve ${H'}_{\!\!\lambda}$ is superspecial if and only if $G(\lambda) = 0$.
As in Proposition \ref{D8iso}, we prove the following criteria to make the isomorphism classification of listed superspecial ${H'}_{\!\!\lambda}$'s efficient:

\begin{Prop}\label{D10iso}
For $\lambda_1,\lambda_2 \in k \smallsetminus \{ 0,1 \}$, the curve ${H'}_{\!\!\lambda_1}\hspace{-0.3mm}$ is isomorphic to the curve ${H'}_{\!\!\lambda_2}\hspace{-0.3mm}$ if and only if $\lambda_1 = \lambda_2$ of $\lambda_1 = 1/\lambda_2$.
\end{Prop}

\begin{proof}
Write the two hyperelliptic curves ${H'}_{\!\!\lambda_1}$ and ${H'}_{\!\!\lambda_2}$ as
\begin{align*}
    {H'}_{\!\!\lambda_1} &: y^2 = (x^5-\alpha^5)(x^5-\alpha^{-5}), \quad {\rm with}\ \alpha^{10} = \lambda_1,\\
    {H'}_{\!\!\lambda_2} &: y^2 = (x^5-\beta^5)(x^5-\beta^{-5}), \quad \hspace{0.1mm}{\rm with}\ \beta^{10} = \lambda_2.
\end{align*}
Similarly to the proof of \cite[Theorem 2.2]{Ishii}, one can show that ${H'}_{\!\!\lambda_1} \cong {H'}_{\!\!\lambda_2}$ if and only if
\[
    \lambda_1 + \frac{1}{\lambda_1} = \alpha^{10} + \frac{1}{\alpha^{10}} = \beta^{10} + \frac{1}{\beta^{10}} = \lambda_2 + \frac{1}{\lambda_2},
\]
which is equivalent to $\lambda_1 = \lambda_2$ or $\lambda_1 = 1/\lambda_2$, as desired.
\end{proof}

Recall from Theorem \ref{thm:app} (especially Figure \ref{fig:aut}) that $\mathrm{Aut}(H_{\lambda}')$ is isomorphic to $\mathbf{D}_{10}$ or $\mathbf{C}_{5} \rtimes \mathbf{D}_4$; we can easily decide which is the case as follows:

\begin{Cor}\label{D10special}
The following statements are true:\vspace{-0.5mm}
\begin{enumerate}
    \item If $\lambda = -1$, then we have that ${\rm Aut}({H'}_{\!\!\lambda}) \cong \mathbf{C}_5 \rtimes \mathbf{D}_{4}$.\vspace{-1mm}
    \item Otherwise, we have that ${\rm Aut}({H'}_{\!\!\lambda}) \cong \mathbf{D}_{10}$.
\end{enumerate}
\end{Cor}\vspace{-3mm}
\begin{proof}
The assertion (1) follows from Remark \ref{Alambda10} and the case {\bf 10} of Table \ref{table:aut}.
For any $\lambda \neq -1$, it follows from Proposition \ref{D10iso} that ${H'}_{\!\!\lambda}$ is not isomorphic to ${H'}_{\!\!-1}$, as desired.
\end{proof}

Moreover, we can show the following results similarly to Theorem \ref{D8thm} and Proposition \ref{D8belong}. 

\begin{Thm}\label{D10thm}
With notation as above, the degree of $G$ is odd if and only if $p \equiv 9 \pmod{10}$, and the number of superspecial hyperelliptic curves $H$ of genus 4 such that ${\rm Aut}(H) \cong {\mathbf D}_{10}$ is equal to\vspace{-1mm}
\[
    \left\{
    \begin{array}{l}
        \frac{1}{2}(\deg{G}-1) \quad {\rm if}\ p \equiv 9 \hspace{-2mm}\pmod{10},\\[1mm]
        \frac{1}{2}\deg{G} \hspace{4.7mm}\qquad {\rm otherwise.}
    \end{array}
    \right.
\]
\end{Thm}

\begin{Prop}\label{D10belong}
All roots of the polynomial $G$ belong to $\mathbb{F}_{p^4}$.
\end{Prop}
{
As in the case treated in Subsection \ref{subsec:D8}, we expect from our computational results that $\deg{G}$ (and thus the number of listed superspecial curves) is bounded by a constant not depending on $p$, see Table \ref{table:new} below.} As a result of the above discussion, we obtain the following algorithm:
\begin{Alg}\label{alg:D10}
    ~
\begin{description}
\item[{\it Input:}] A rational prime $p\geq 7$.
\item[{\it Output:}] A list $\mathcal{L}$ of representatives for the isomorphism classes of superspecial genus-$4$ hyperelliptic curves $H$ over $\overline{\mathbb{F}_p}$ with $\mathrm{Aut}(H) \cong \mathbf{D}_{10}$.
\end{description}
\begin{description}
\item[Step 1.] Compute $\beta_{3p-i}$ and $\beta_{4p-j}$ by using Proposition \ref{alpha} {(see also the first paragraph in the proof of Theorem \ref{D8comp})}, where $(i,j)$ denotes a pair of integers defined in \eqref{D10ij}.
\item[Step 2.] {Compute the polynomial $G := {\rm gcd}(\beta_{3p-i},\beta_{4p-j})$ and its roots $\lambda \in \mathbb{F}_{p^4}$.}
\item[Step 3.] For each $\lambda \neq -1$ obtained in Step 2, if not $H_{1/\lambda}'$ already belongs to $\mathcal{L}$, then set $\mathcal{L} \leftarrow \mathcal{L} \hspace{0.5mm}\cup \{H_\lambda'\}$. Finally, output a list $\mathcal{L}$.
\end{description}
\end{Alg}

The complexity of this algorithm can be estimated as follows:

\begin{Thm}\label{D10comp}
We can compute the number of superspecial genus-$4$ hyperelliptic curves $H$ with $\mathrm{Aut}(H) \cong \mathbf{D}_{10}$ in $\tilde{O}(p)$ operations in $\mathbb{F}_{p}$, and we can list all of their isomorphism classes in $\tilde{O}(p+d^2)$ operations in $\mathbb{F}_{p^4}$ with $d := \deg{G}=O(p)$.
{
Moreover, the complexity of listing up the isomorphism classes becomes $\tilde{O}(p)$ operations in $\mathbb{F}_{p^4}$, if $\deg{G}$ is bounded by a constant not depending on $p$.}
\end{Thm}

\begin{proof}
This can be shown similarly to the proof of Theorem \ref{D8comp}.
\end{proof}

%% file: sec4.tex
\section{Computational results}\label{sec:comp}

We implemented {Algorithms \ref{alg:D4}, \ref{alg:D8}, and \ref{alg:D10}} on {Magma V2.26-10 on a computer with Intel(R) Xeon(R) Gold 6130 CPU at 2.10GHz with 16 cores (32 threads) and 768GB memory.}
Executing the implemented algorithm, we obtain Theorem \ref{ThmB} in Section \ref{sec:intro}.
We here explain details dividing into enumeration results and search ones.

\newpage

\renewcommand{\arraystretch}{0.9}
\begin{table}[H]
\centering{
\caption{The total number of isomorphism classes of superspecial hyperelliptic curves $H$ of genus $4$ with $\mathrm{Aut}(H) \supset \mathbf{D}_4$ for each $p$ with $19 \leq p< 500$, and their classification by types of automorphism groups.
We denote $\mathbf{C}_{16} \rtimes \mathbf{C}_2$ and $\mathbf{C}_{5} \rtimes \mathbf{D}_4$ by $\mathbf{G}_{32}$ and $\mathbf{G}_{40}$ respectively.
}
\label{table:new}\vspace{3mm}
\hspace{-7mm}\begin{minipage}{58mm}
\scalebox{0.93}{
\begin{tabular}{c||c|c|c|c|c|c}\hline
$p$ & All & $\mathbf{D}_4$ & $\mathbf{D}_8$ & $\mathbf{D}_{10}$ & $\mathbf{G}_{32}$ & $\mathbf{G}_{40}$\\\hline
 19 &  2 &  1 & 0 & 0 & 0 & 1\\\hline
 23 &  1 &  1 & 0 & 0 & 0 & 0\\\hline
 29 &  2 &  1 & 0 & 0 & 0 & 1\\\hline
 31 &  3 &  1 & 1 & 0 & 1 & 0\\\hline
 37 &  4 &  3 & 1 & 0 & 0 & 0\\\hline
 41 &  2 &  0 & 0 & 1 & 1 & 0\\\hline
 43 &  2 &  2 & 0 & 0 & 0 & 0\\\hline
 47 &  5 &  4 & 0 & 0 & 1 & 0\\\hline
 53 &  6 &  6 & 0 & 0 & 0 & 0\\\hline
 59 &  4 &  2 & 0 & 1 & 0 & 1\\\hline
 61 &  4 &  2 & 0 & 1 & 0 & 0\\\hline
 67 &  4 &  4 & 0 & 0 & 0 & 0\\\hline
 71 &  6 &  4 & 1 & 1 & 0 & 0\\\hline
 73 &  6 &  5 & 0 & 0 & 1 & 0\\\hline
 79 &  9 &  7 & 0 & 0 & 1 & 1\\\hline
 83 & 10 & 10 & 0 & 0 & 0 & 0\\\hline
 89 &  7 &  4 & 0 & 1 & 1 & 1\\\hline
 97 &  4 &  4 & 0 & 0 & 0 & 0\\\hline
101 &  7 &  4 & 2 & 1 & 0 & 0\\\hline
103 & 10 &  9 & 1 & 0 & 0 & 0\\\hline
107 & 11 & 11 & 0 & 0 & 0 & 0\\\hline
109 &  7 &  5 & 0 & 1 & 0 & 1\\\hline
113 & 13 & 13 & 0 & 0 & 0 & 0\\\hline
127 & 14 & 13 & 0 & 0 & 1 & 0\\\hline
131 & 12 & 12 & 0 & 0 & 0 & 0\\\hline
137 & 11 &  9 & 1 & 0 & 1 & 0\\\hline
139 & 19 & 16 & 0 & 2 & 0 & 1\\\hline
149 & 15 & 13 & 0 & 1 & 0 & 1\\\hline
151 & 10 & 10 & 0 & 0 & 0 & 0\\\hline
157 & 16 & 13 & 3 & 0 & 0 & 0\\\hline
163 &  8 &  8 & 0 & 0 & 0 & 0\\\hline
167 & 19 & 17 & 2 & 0 & 0 & 0\\\hline
173 & 11 & 11 & 0 & 0 & 0 & 0\\\hline
179 & 23 & 19 & 0 & 3 & 0 & 1\\\hline
181 & 14 & 10 & 3 & 1 & 0 & 0\\\hline
191 & 37 & 31 & 3 & 2 & 1 & 0\\\hline
193 & 18 & 18 & 0 & 0 & 0 & 0\\\hline
197 & 12 &  9 & 3 & 0 & 0 & 0\\\hline
199 & 26 & 22 & 2 & 1 & 0 & 1\\\hline
211 & 22 & 20 & 0 & 2 & 0 & 0\\\hline
223 & 37 & 35 & 1 & 0 & 1 & 0\\\hline
227 & 15 & 15 & 0 & 0 & 0 & 0\\\hline
229 & 13 & 10 & 0 & 2 & 0 & 1\\\hline
233 & 13 & 11 & 1 & 0 & 1 & 0\\\hline
\end{tabular}}
\end{minipage}
\hspace{5mm}
\begin{minipage}{58mm}
\scalebox{0.93}{
\begin{tabular}{c||c|c|c|c|c|c}\hline
$p$ & All & $\mathbf{D}_4$ & $\mathbf{D}_8$ & $\mathbf{D}_{10}$ & $\mathbf{G}_{32}$ & $\mathbf{G}_{40}$\\\hline
239 & 33 & 30 & 0 & 1 & 1 & 1\\\hline
241 & 20 & 17 & 1 & 2 & 0 & 0\\\hline
251 & 27 & 27 & 0 & 0 & 0 & 0\\\hline
257 & 17 & 16 & 1 & 0 & 0 & 0\\\hline
263 & 35 & 34 & 1 & 0 & 0 & 0\\\hline
269 & 39 & 32 & 3 & 3 & 0 & 1\\\hline
271 & 48 & 40 & 5 & 2 & 1 & 0\\\hline
277 & 28 & 28 & 0 & 0 & 0 & 0\\\hline
281 & 14 & 11 & 1 & 1 & 1 & 0\\\hline
283 & 24 & 24 & 0 & 0 & 0 & 0\\\hline
293 & 19 & 19 & 0 & 0 & 0 & 0\\\hline
307 & 22 & 22 & 0 & 0 & 0 & 0\\\hline
311 & 44 & 42 & 1 & 1 & 0 & 0\\\hline
313 & 30 & 28 & 1 & 0 & 1 & 0\\\hline
317 & 26 & 26 & 0 & 0 & 0 & 0\\\hline
331 & 36 & 36 & 0 & 1 & 0 & 0\\\hline
337 & 22 & 22 & 0 & 0 & 0 & 0\\\hline
347 & 16 & 16 & 0 & 0 & 0 & 0\\\hline
349 & 35 & 32 & 0 & 2 & 0 & 1\\\hline
353 & 30 & 29 & 1 & 0 & 0 & 0\\\hline
359 & 56 & 53 & 0 & 2 & 0 & 1\\\hline
367 & 45 & 43 & 1 & 0 & 1 & 0\\\hline
373 & 28 & 28 & 0 & 0 & 0 & 0\\\hline
379 & 36 & 33 & 0 & 2 & 0 & 1\\\hline
383 & 58 & 56 & 1 & 0 & 1 & 0\\\hline
389 & 39 & 36 & 0 & 2 & 0 & 1\\\hline
397 & 27 & 27 & 0 & 0 & 0 & 0\\\hline
401 & 41 & 39 & 1 & 1 & 0 & 0\\\hline
409 & 35 & 30 & 1 & 2 & 1 & 1\\\hline
419 & 45 & 39 & 0 & 5 & 0 & 1\\\hline
421 & 36 & 32 & 4 & 0 & 0 & 0\\\hline
431 & 64 & 61 & 0 & 2 & 1 & 0\\\hline
433 & 41 & 40 & 1 & 0 & 0 & 0\\\hline
439 & 64 & 56 & 2 & 5 & 0 & 1\\\hline
443 & 37 & 37 & 0 & 0 & 0 & 0\\\hline
449 & 28 & 27 & 0 & 0 & 0 & 1\\\hline
457 & 43 & 42 & 0 & 0 & 1 & 0\\\hline
461 & 33 & 29 & 0 & 4 & 0 & 0\\\hline
463 & 47 & 44 & 2 & 0 & 1 & 0\\\hline
467 & 36 & 36 & 0 & 0 & 0 & 0\\\hline
479 & 73 & 68 & 1 & 2 & 1 & 1\\\hline
487 & 40 & 38 & 1 & 0 & 0 & 0\\\hline
491 & 43 & 42 & 0 & 1 & 0 & 0\\\hline
499 & 49 & 47 & 0 & 1 & 0 & 1\\\hline
\end{tabular}}
\end{minipage}}
\end{table}

\paragraph{Enumeration results.}
As indicated in Theorem \ref{ThmB}, our {enumeration} results involve
the number of isomorphism classes of superspecial hyperelliptic curves $H$ of genus $4$ in each type of automorphism group containing $\mathbf{D}_4$ for $19 \leq p < 500$, and they are summarized in Table \ref{table:new}.
Note that there is no such a curve in characteristic $p \leq 17$.
We see from Table \ref{table:new} that the number of superspecial $H$'s with $\mathrm{Aut}(H) \cong \mathbf{D}_4$ would follow $O(p)$, as estimated in Section \ref{subsec:D4}.
On the other hand, the number of superspecial $H$'s with automorphism group of other type might be bounded by a constant not depending on $p$.

We can easily increase the upper bound on $p$ in Table \ref{table:new}:
We have succeeded in finishing the computation of $\overline{\mathbb{F}_p}$-isomorphism classes for all $19 \leq p < 500$ within a few day in total, on the {computer} {described above.
For example, the execution for $p=499$ took only {1,441} seconds.
More specifically, the time taken for Step 1 was 971 seconds, while that for Steps 2--3 was 470 seconds.
}

{
\paragraph{Search results.}
Search here means that we find a single superspecial hyperelliptic curve $H$ of genus $4$ with $\mathrm{Aut}(H) \supset \mathbf{D}_4$.
For this, we construct and implement an algorithm based on the following strategy:
\begin{enumerate}
    \item {
    It follows from Corollary \ref{spssp} that there does exist a superspecial $H$ with $\mathrm{Aut}(H) \cong \mathbf{C}_{16} \rtimes \mathbf{C}_2$ (resp.\ $\mathrm{Aut}(H) \cong \mathbf{C}_{5} \rtimes \mathbf{D}_4$) if $p \equiv -1,9 \pmod{16}$ (resp.\ $p \equiv 9 \pmod{10}$).
    In these cases,} we need not to execute any of Algorithms \ref{alg:D4}, \ref{alg:D8}, nor \ref{alg:D10}.
    {Otherwise, we proceed with (2) and (3) in order.}
    \item We try to find a superspecial $H$ with $\mathrm{Aut}(H) \cong \mathbf{D}_8$ or $\mathrm{Aut}(H) \cong \mathbf{D}_{10}$ by Algorithms \ref{alg:D8} and \ref{alg:D10}.
    For our purpose, it suffices to compute one root of $F$ (resp.\ $G$) in Step 2 of Algorithm \ref{alg:D8} (resp.\ Algorithm \ref{alg:D10}).
    \item When we fail to produce a superspecial $H$ by (2), we next execute Algorithm \ref{alg:D4}.
    For the efficiency of the search, we modify Algorithm \ref{alg:D4} as follows:
    Once a single genus-$2$ superspecial curve $C$ is generated in Step 1 of Algorithm \ref{alg:D4}, we check whether there exists $C'$ with $C \cong C'$ such that the normalization $H$ of $C \times_{\mathbb{P}^1}C'$ is a genus-$4$ superspecial hyperelliptic curve with $\mathrm{Aut}(H) \cong \mathbf{D}_4$, by using Lemma \ref{lem:D4}.
\end{enumerate}}

Executing the search algorithm on Magma, we succeeded in finding a superspecial hyperelliptic curve $H$ of genus $4$ with $\mathrm{Aut}(H) \supset \mathbf{D}_4$ for every $p$ with $19 \leq p < 7000$, as in \cite{OKH}.
We have succeeded in finishing the search for all $p$ in the interval within a week in total, on the same {computer} {used for the enumeration.
This upper bound can be also updated easily:
\begin{itemize}
    \item Among $19 \leq p < 7000$, the maximal time for finding a single superspecial curve was 19,144 seconds ($\approx$ 5.4 hours) for $p=6763$.
    In this case, we failed to find such a curve by (2) of the above strategy, and most of the time was spent for (3) (the time spent for (3) was 19,119 seconds), whose complexity is bounded by that of Algorithm \ref{alg:D4} (i.e., {${O}(p^3)$} as estimated in Section \ref{subsec:D4}).
    \item Even in the case where we fail to find any superspecial curve in (2), there are many large $p$ for which such a curve was found in shorter time (e.g., only 6.8 seconds spent for proceeding with (3) when $p=6373$).
    \item Fortunately if we can find a single superspecial curve only by (2) (not proceeding with (3)), the time spent for (2) is expected to follow {$\tilde{O}(p)$} (cf.\ Theorems \ref{D8comp} and \ref{D10comp}), and thus in most cases it would be quite shorter than that for (3).
    For example, we found such a curve by (2) in 27.6 seconds for $p=6997$.
\end{itemize}}

{
From our search results, we obtain the following conjecture:}

\begin{Conj}\label{conj}
\begin{enumerate}
    \item There exists a superspecial hyperelliptic curve $H$ of genus $4$ with {$\mathrm{Aut}(H) \cong \mathbf{D}_4$} in arbitrary characteristic $p$ with $p \geq 19$ and $p \neq 41$.
    \item There does not exist any superspecial hyperelliptic curve $H$ of genus $4$ with $\mathrm{Aut}(H) \supset \mathbf{D}_{8}$ if $p \equiv 3 \pmod{8}$.
    \item {There does not exist any superspecial hyperelliptic curve $H$ of genus $4$ with $\mathrm{Aut}(H) \supset \mathbf{D}_{10}$ if $p \equiv 3,7 \pmod{10}$.}
    \end{enumerate}
\end{Conj}

%% file: sec5.tex
\section{Concluding remarks}

We proposed an algorithm (Algorithm \ref{alg:D4}) to enumerate superspecial hyperelliptic curves $H$ of genus $4$ with $\mathrm{Aut}(H) \supset \mathbf{D}_4$, as a simpler version of the algorithm provided in \cite{OKH}.
{All the computations of the algorithm are done in $\mathbb{F}_{p^4}$, and the total complexity is expected to be $O(p^3)$ operations in $\mathbb{F}_{p^4}$.}
More efficient algorithms (Algorithms \ref{alg:D8} and \ref{alg:D10}) specific to the case where $\mathrm{Aut}(H) \supset \mathbf{D}_8$ or $\mathrm{Aut}(H) \supset \mathbf{D}_{10}$ were also constructed by utilizing Cartier-Manin matrices and Gauss' hypergeometric series.
{Their complexities are both $\tilde{O}(p^2)$ {(in practice $\tilde{O}(p)$)} operations in $\mathbb{F}_{p^4}$, which would become $\tilde{O}(p)$ operations in $\mathbb{F}_p$ when one computes only the number of isomorphism classes}.
By executing the algorithms on Magma, we succeeded in enumerating all the isomorphism classes of superspecial hyperelliptic curves $H$ of genus $4$ with $\mathrm{Aut}(H) \supset \mathbf{D}_4$ for every $p$ from $19$ to $500$.
{We also succeeded in finding such a curve for every $p$ from $19$ to $7000$.
From our computational results, we obtain a conjecture (Conjecture \ref{conj}) that there exists a superspecial hyperelliptic curve $H$ of genus $4$ with $\mathrm{Aut}(H) \supset \mathbf{D}_4$ in arbitrary characteristic $p$ with $p \geq 19$.}

A future work is to theoretically prove Conjecture \ref{conj}.
Constructing an efficient algorithm in the case where $\mathrm{Aut}(H) \cong \mathbf{C}_4$ or $\mathrm{Aut}(H) \cong \mathbf{Q}_8$ would be also an interesting problem.

%% file: secA.tex
\renewcommand{\thesection}{\Alph{section}}
\setcounter{section}{0}
\theoremstyle{definition}
\newtheorem{defi}{Definition}[section]
\newtheorem{lem}[defi]{Lemma}
\newtheorem{prop}[defi]{Proposition}
\newtheorem{thm}[defi]{Theorem}
\newtheorem{cor}[defi]{Corollary}

\section{{Some facts on Gauss' hypergeometric series}}
{In this appendix, we correct some facts related to Gauss' hypergeometric series, which are used in the main contents of this paper.}
Let $K$ be an arbitrary field of characteristic $p \geq 5$.
\begin{lem}
For an integer $d > 0$, the following statements are true:
\begin{enumerate}
\item If $0 \leq b \leq a$, then we have
\[
    \sum_{n=0}^d \binom{a}{n}\binom{a}{b-n} z^n = \binom{a}{b}G^{(d)}(-a,-b,1+a-b\,;z).
\]
\item If $a \leq b \leq 2a$, then we have
\[
    \sum_{n=0}^d \binom{a}{n}\binom{a}{b-n} z^n = z^{b-a} \cdot \binom{a}{2a-b}G^{(d)}(-a,b-2a,1+b-a\,;z).
\]
\end{enumerate}
\end{lem}
\begin{proof}
This follows from a simple computation.
\end{proof}

We consider an element $\lambda \neq 0,1$ of the field $K$. Let $\gamma_k$ be the $x^k$-coefficient of the polynomial $\{(x-1)(x-\lambda)\}^e$ with $e := (p-1)/2$.
\begin{lem}
For an integer $0 \leq k \leq e$, we have
\[
    \gamma_k = (-1)^{p-1-k}\sum_{n=0}^k\binom{(p-1)/2}{n}\binom{(p-1)/2}{k-n}\lambda^{e-k+n}.
\]
\end{lem}
\begin{proof}
Let $e := (p-1)/2$, then we can compute
\[
    (x-1)^e= \sum_{i=0}^e\!\binom{e}{i}(-1)^{e-i}x^i, \quad (x-\lambda)^e= \sum_{j=0}^e\!\binom{e}{j}(-\lambda)^{e-j}x^j,
\]
and hence we obtain
\[
    \{(x-1)(x-\lambda)\}^{(p-1)/2} = \sum_{k=0}^{p-1}\Biggl(\,\sum_{i+j=k}\!\binom{e}{i}\!\binom{e}{j}\hspace{-0.3mm}(-1)^{e-i}(-\lambda)^{e-j}\!\Biggr)x^k.
\]
Comparing the $x^k$-coefficients of both sides, we have this lemma.
\end{proof}

{Combining the above two lemmas together}, we can write $\gamma_k$ as a truncation of Gauss' hypergeometric series:
\begin{prop}\label{gamma}
The following statements are true:
\begin{enumerate}
    \item If $0 \leq k \leq e$, then we have
    \[
        \gamma_k \equiv \lambda^{e-k} \cdot (-1)^k\binom{e}{k}G^{(k)}(1/2,-k,1/2-k\,;\lambda) \pmod{p}.
    \]
    \item If $e \leq k \leq p-1$, then we have
    \[
        \gamma_k \equiv (-1)^k\binom{e}{p-1-k}G^{(k)}(1/2,1+k,3/2+k\,;\lambda) \pmod{p}.
    \]
\end{enumerate}
\end{prop}
\begin{proof}
It follows from the previous lemmas.
\end{proof}
\begin{cor}\label{div}
For all $0 \leq k \leq e$, we have that $\gamma_{e-k} \equiv \gamma_{e+k}\lambda^k \pmod{p}$.
\end{cor}
\begin{proof}
It follows from Corollary \ref{gamma} that
\begin{align*}
    \gamma_{e-k} &\equiv \lambda^k \cdot (-1)^{e-k}\binom{e}{e-k}G^{(e-k)}(1/2,-e+k,1/2-e+k\,;\lambda)\\
    & \equiv \lambda^k \cdot (-1)^{e+k}\binom{e}{e-k}G^{(e+k)}(1/2,1+e+k,3/2+e+k\,;\lambda)\\
    & \equiv \gamma_{e+k}\lambda^k \pmod{p}.
\end{align*}
Here, we use the fact that
\[
    G^{(e+k)}(1/2,-e+k,1/2-e+k\,;\lambda) = G^{(e-k)}(1/2,-e+k,1/2-e+k\,;\lambda)
\]
since $\lambda^n$-terms are vanished $\hspace{-1.7mm}\pmod{p}$ for all $e-k < n \leq e+k$.
\end{proof}